\documentclass[letterpaper, 10 pt, conference]{ieeeconf}  

\IEEEoverridecommandlockouts                              
\newcommand{\cR}{\mathbb{R}}

\def \ds {\displaystyle}

\def \U {{\mathcal  U}}

\newtheorem{theorem}{Theorem}[section]
\newtheorem{lemma}[theorem]{Lemma}
\newtheorem{proposition}[theorem]{Proposition}
\newtheorem{corollary}[theorem]{Corollary}
\newtheorem{definition}[theorem]{Definition}

\newtheorem{remark}[theorem]{Remark}
\def\bel{\begin{equation}\label}
\def\eeq{\end{equation}}
\def\cS{\mathfrak{T}}

\def \Ba {\mathbb{B}}

\def \ds{\displaystyle}
\def\vsm{\vskip0.3truecm}

\def \va{\nu}

 \usepackage{amsmath} 
 \usepackage{amssymb}  
 \usepackage[usenames, dvipsnames]{color}

\title{\LARGE \bf
Necessary conditions involving Lie brackets \\for impulsive optimal control problems*
}

\author{M. Soledad Aronna$^{1}$, Monica Motta$^{2}$ and Franco Rampazzo$^{3}$
 \thanks{*This research is partially supported by the  Padua University grant SID 2018 ``Controllability, stabilizability and infimun gaps for control systems'', prot. BIRD 187147; by the ``National Group for Mathematical Analysis, Probability and their Applications"  (GNAMPA-INdAM) (Italy);    by the E.U. under the 7th Framework Program  - Grant agreement SADCO,  by CAPES, CNPq, and FAPERJ (Brazil) through the ``Jovem Cientista de Nosso Estado'' Program and by the von Humboldt Foundation (Germany).  }
\thanks{$^{1}$M.S. Aronna,  Escola de Matem\'atica Aplicada,  
 Funda\c c\~ ao 
Getulio Vargas,  Rio de Janeiro 22250-900, Brazil
        {\tt\small soledad.aronna@fgv.br}}
\thanks{$^{2}$M. Motta, Dipartimento di Matematica,
Universit\`a di Padova,  Padova  35121, Italy
        {\tt\small motta@math.unipd.it}}%
        \thanks{$^{3}$F. Rampazzo, Dipartimento di Matematica,
Universit\`a di Padova,  Padova  35121, Italy
        {\tt\small rampazzo@math.unipd.it}}%
 }

\begin{document}

\maketitle

\thispagestyle{empty}

\pagestyle{empty}

\begin{abstract}
We obtain higher order  necessary conditions for a minimum of a Mayer optimal control problem connected with a nonlinear, control-affine system, where the controls range on an m-dimensional Euclidean space. Since the allowed velocities are unbounded and the absence of  coercivity assumptions makes big speeds quite likely, minimizing sequences happen to converge toward “impulsive”, namely discontinuous, trajectories. As is known, a distributional approach does not make sense in such a nonlinear setting, where instead a suitable embedding in the graph space is needed. We will illustrate how the chance of using impulse  perturbations makes it possible to derive a Higher Order Maximum Principle which includes both the usual needle variations (in space-time) and  conditions involving iterated Lie brackets. An example, where  a third  order necessary condition  rules out the optimality of a given extremal,  concludes the paper.
\end{abstract}

\section{INTRODUCTION}

In this paper we aim to investigate necessary optimality conditions for an optimal process of the following  minimum problem:
\begin{equation*}\label{intro1}
(P)
\left\{
\begin{split}
&\qquad\mbox{Minimize } \Psi(T, x(T))
\\ &
\ds\mbox{over the set of processes} \ (T,u,x) \   \mbox{satisfying }  \\ 
&\displaystyle\frac{dx}{dt} \,=\, f(x) + \sum_{i=1}^{m}g_{i}(x)u^i,    \\
&x(0)=\check x, \quad \big(T, x(T)\big) \in \cS.
\end{split}
\right.
\end{equation*}
 The {\it target } $\cS$ is given by
 \begin{multline*}
 \cS:=\{(t,x):  \ \varphi_i(t,x)\le 0, \ \psi_j(t,x)=0,   \\
   i=1,\dots, r_1, \ j=1,\dots,r_2\}.
  \end{multline*}
 where $\varphi_i,\, \psi_j:\cR\times\cR^{n}\to\cR$ are functions  of class $C^1$. 
We assume  that the state variable ranges over $\cR^n$ and that the control $u$ takes values in   $\cR^m$: in particular, $u$ is  allowed to be { \it unbounded}.  The cost function  $\Psi:\cR\times\cR^n  \to\cR$ is assumed of class $C^1$, while $f$, $g_i:\cR^n\to\cR^n$   are  the vector fields   of class $C^\infty$. However,  we refer to Remark \ref{Rext} for  comments on the possibility of drastically reducing  these regularity hypotheses and replacing  the state space $\cR^n$ with a manifold,  as done in \cite{AMR19}.
 
 Let us point out  that  problem $(P)$ has an {\it impulsive} character, i.e. minimizing  sequences generally fail to converge to  an absolutely continuous path  and, in fact,  they may happen to approach  a discontinuous  path. 
It is well-known that,  because of the nonlinearity of the dynamics,   a   measure-theoretical approach, with $u$ to be interpreted as a Radon measure, does not verify basic well-posedness conditions  \cite{Haj85}. Different but substantially equivalent approaches take care of this crucial point (see, among others,   \cite{BR88,MR95, MiRu03}).  We choose here to adopt the so-called  {\em graph-completion} point of view and  embed the original problem into the  {\it space-time} problem  ($P^{e}$)    below:   the extended state variable is now $(y^0,y):=(t,x)$, and the extended trajectories are $(t,x)$-paths  which are (reparameterized) $C^0$-limits of graphs of the original trajectories.

 More precisely, we consider  the optimization problem
\begin{equation*}\label{extendedintro}
(P^{e})
\left\{
\begin{split}
& \qquad \mbox{Minimize } \Psi(y^0(S), y(S))\\
& \ds\mbox{over the set of processes} \, (S,w^0,w,y^0,y) \,  \mbox{verifying }  \\ 
& \ds \frac{d{y^0}}{ds}  = w^0,\\ 
&\ds \frac{dy}{ds} = f(y)w^0+ \sum_{i=1}^{m}g_{i}( y)w^i, \\
&(y^0,y)(0)={(0,\check x)}, \qquad (y^0(S), y(S)) \in \cS,
\end{split}\right. 
\end{equation*}
where the  controls $(w^0,w)$ are functions from a pseudo-time interval $[0,S]$ into the set 
 \bel{DS+}
\mathcal{W} :=\{(w^0,w)\in\cR  \times \cR^m: \ w^0\ge0, \   w^0+|w|=1\} .
\eeq
 Notice that, unlike the controls $u$, the control pairs $(w^0,w)$ are now bounded. A process $(T,u,x)$ of the original system is identified with a process $(S,,w^0,w, y^0,y)$ of the space-time system through the reparameterization
 \begin{equation}\label{sigma}
 \begin{array}{l}
 \sigma(t):=\int_0^t (1+|u(\tau)|)\,d\tau, \quad  y^0:=\sigma^{-1}:[0,S]\to[0,T], \\ [1.5ex]
\ds w^0(s):=(1+|u( y^0(s))|)^{-1},  \quad  w(s):=w^0(s)\,u(y^0(s)),\\ [1.5ex]
 y(s):=x(y^0(s)),
 \end{array}
\end{equation}
while the actual {\it impulsive} processes --namely the ones that are not reparameterizations of original processes--  are the five-tuples  $(S,w^0,w, y^0,y)$  with $w^0=0$ on some  non trivial subinterval
$[s_1,s_2]\subseteq [0,S]$.   Unlike the original problem  $(P)$, the extended problem often admits   an optimal
process (provided the target can be reached by one trajectory), so that it is natural to look for necessary conditions for the extended problem $(P^e)$. 

The first part of the paper (Section \ref{SecProblem}) is devoted to the  consistency of minimum problems $(P)$ and $(P^e)$, in their local version. Actually, it turns out that a process  $(\bar T,\bar u,\bar x)$ of the original system is locally optimal  for the original problem $(P)$ if and only if its space-time representation $(\bar S,\bar w^0,\bar w,\bar  y^0,\bar y)$ is locally optimal for the restriction of the  extended problem $(P^e)$ to processes  $(S,w^0,w,y^0,y)$ verifying $w^0>0$ almost everywhere.
 
In Theorem \ref {PMPeho} we state an abridged version of a  Higher Order Maximum Principle for $(P^e)$, where   necessary optimality conditions   involving iterated Lie brackets of the non-drift vector fields $\{g_1,\ldots,g_m\}$ are presented.   In particular, this result generalizes to impulsive trajectories a result that, for minimum time problems, has been established  for absolutely  continuous processes with unbounded controls (see \cite{ChiSt16}).  The detailed proof  of the  result's main point is rather  long and technical, and is provided --under much weaker regularity hypotheses--- in \cite{AMR19}. Instead, here we just give   some hints of the idea  lying behind the stated higher order conditions. Further relations, which involve the drift $f$, are derived in Corollary \ref{COR}.
 The paper concludes with an Example, in Section  \ref{SectionExample},  where the optimality of a space-time process  verifying the standard maximum principle  is ruled out by our higher order  conditions.

Several papers, an incomplete list of which  includes \cite{SV97, PS00, MiRu03}, deal with  First Order Maximum Principles for impulsive systems. As for  {\it higher-order}  necessary  conditions ---see e.g.  \cite{bonnans, FR, knobloch, krener, sussmann, SLed12} for the bounded control case---  we are aware only of results  for the  {\it  commutative case}, i.e. when $[g_i,g_j]\equiv0$ for all  $i,j=1,\ldots,m$, and up to the  second order (see e.g.  \cite{AKPC18,ADP05,D94}).  Instead, our Higher Order Maximum Principle is established for the generic, non-commutative, case, and  involves iterated brackets of any order.

\subsection{Notation and definitions}
Let $N$ be a natural number. For every $i\in\{1,\dots,N\}$, we write ${\bf e}_i$  for the $i$-th element of the  canonical basis of $\cR^N,$  
$\Ba_N({\check x})$ for the closed ball $\{x\in \cR^N : |x-{\check x}| \leq 1\}$,  and $\Ba_N$ when ${\check x}=0.$ $\partial \Ba_N:=\{x \in\cR^N: \  |x|=1\}$.   A subset $K\subseteq {\cR^N}$  is called  {\it cone} if $\alpha x\in K$ whenever $\alpha>0$, $x\in K$. 
Given a real interval $I$ and $X\subseteq \cR^N,$ we write $AC(I,X)$ for the space of absolutely continuous functions, $C^0(I,X)$ for the space of continuous functions,   $L^1(I,X)$  for the space of $L^1$-functions, and $L^\infty(I,X)$   for space of measurable, bounded functions, respectively, defined on $I$ with values in $X.$ As customary, we shall use  $\|\cdot\|_\infty$, $\|\cdot\|_1$ to denote the sup-norm and the $L^1$-norm, respectively, where  domain and codomain are omitted  when obvious.
 
  The Lie bracket of two vector fields $F_1,F_2$ is the vector field $[F_1,F_2]$ defined by
$$
[F_1,F_2](x):=DF_2(x)\cdot F_1(x) - DF_1(x)\cdot F_2(x),
$$
 where $D$ denotes differentiation.  By repeating the bracketing procedure  we obtain the so-called iterated brackets. 
 
  For a given  real interval $I$, let us consider the {\it $L^1$-norm operator }$\va:L^1(I,\cR^m)\to AC(I,[0,+\infty))$ defined by
\begin{equation}\label{var}
\va[u](t):=  \int_0^t|u(\tau)|\,d\tau, \qquad \text{for } t\in I.
\end{equation}

\section{THE OPTIMIZATION PROBLEM AND ITS EXTENSION}\label{SecProblem}
In this section we introduce the optimization problem over $L^1$ controls  and its embedding in an impulsive problem in detail. 
 
\subsection{The original optimal control problem}
We define the {\em set of strict sense controls} as
$$
\U:=
\bigcup_{T>0} \{T\}\times  L^1([0,T],\cR^m). 
$$

\begin{definition}
 For any $(T,u)\in\U$ we say that $(T,u, x)$ is a {\em strict-sense process} if  $x$ is  the unique Carath\'eodory solution to
 \bel{E2}
\left\{\begin{array}{l}
\ds \frac{dx}{dt}(t) = f(x(t)) + \ds\sum_{i=1}^m  g_i(x(t)) {u^i}(t)\\
 x(0)=\check x \end{array}\right.
\eeq corresponding to the control $u$ and defined on $[0,T]$. \footnote{Under our assumptions on the control system, for any strict-sense control $(T,u)$, there exists a unique solution of \eqref{E2} which is defined in general on a maximal interval of definition $[0,\tau)\subseteq [0,T]$.}
Furthermore, we say that a process $(T,u,x)$ is {\em  feasible} if it agrees with the final constraint, i.e. $(T, x(T))\in \cS$.  
 \end{definition}

Let us fix an integer $q$ and let us  define a distance by setting for all $\tau_1$, $\tau_2\in (0,+\infty)$ and  for any pair $(z_1,z_2)\in C^0([0,\tau_1],\cR^q)\times C^0(\tau_2 ([0,\tau_2],\cR^q)$,
 \begin{equation}\label{dinfty}
{\rm d}\big((\tau_1,z_1),(\tau_2,z_2)\big) :=  
 | \tau_1 - \tau_2|+  \|\tilde z_1- \tilde z_2\|_{{\infty}},
  \end{equation} 
 where, for every map  $z\in C^0([0,\tau],\cR^q)$ we have used $\tilde z$ to denote its continuous constant extension to $[0,+\infty)$. 
   
\begin{definition}\label{dmin} We say that a feasible  strict sense process $(\bar T, \bar u,  \bar x)$  is a {\em strict sense $L^{\infty}$-local minimizer} of $(P)$ if  there exists $\delta >0$ such that
\begin{equation}\label{min1}
\Psi(\bar T,\bar x(\bar T)) \,\leq\, \Psi(T, x(T))
\end{equation}
for every feasible strict sense process $(T, u, x)$ 
verifying
$$
{\rm d}\Big((T, x, \va[u]), (\bar T,  \bar x, \va[\bar u]) \Big)<\delta.
$$
 If  relation \eqref{min1} is satisfied for all admissible strict sense processes,  
 we say that $(\bar T, \bar u, \bar x)$ is a {\em global strict sense minimizer}.  
 \end{definition}

\subsection{The  space-time optimal control problem}\label{Subs3.2}
Define the {\em set of space-time controls}  
\bel{ACc}
\ds W:=
\bigcup_{S>0} \{S\}\times L^\infty([0,S],\mathcal{W}), 
\eeq 
where $\mathcal{W}$ is as in  \eqref{DS+}.
\begin{definition}
 For any  space-time control $(S, w^0,w)\in W,$  we say that $(S,w^0,w,y^0,y)$ is a {\em space-time process} if  $(y^0,y)$ is  the unique Carath\'eodory solution to
 \begin{equation}
 \label{extended}
 \left\{
\begin{split}
\frac{d{y^0}}{ds}  &= w^0,   \\
\frac{dy}{ds} & = f( y)w^0+ \sum_{i=1}^{m}g_{i}( y)w^i, \ \  \mbox{ a.e. } s \in [0, S],\\
 (y^0,y &)(0)=(0,\check x),  
\end{split}
\right. \end{equation}
 corresponding to the control $(w^0,w)$ and defined on $[0,S]$. 
As before, we say that a space-time process $(S,w^0,w,y^0,y)$ is {\em  feasible} if $(y^0(S),y(S))\in \cS$. 
 \end{definition}
 \vsm
 \begin{definition}\label{emin} A feasible space-time process   $(\bar S, \bar w^0,\bar w, \bar y^0,\bar  y)$   is said to be an  {\em  $L^{\infty}$-local minimizer} for the space-time problem $(P^e)$ if 
  there exists $\delta >0 $ such that
\begin{equation}
\label{min}
\begin{array}{l} 
\ds\Psi\big((\bar y^0,\bar y)(\bar S)\big) \leq  \Psi\big((y^0,y)(S)\big) 
\end{array}
\end{equation}
 for all feasible space-time processes $(S,w^0,w,y^0,y)$ satisfying
\begin{equation}
\label{close}
{\rm d}\Big((y^0(S), y ,\va[w]) , ( \bar y^0(\bar S), \bar y, \va[\bar w])\Big)<\delta.
\end{equation}
If \eqref{min} is satisfied for all feasible space-time processes, 
we say that  
 $(\bar S, \bar w^0,\bar w,  \bar y^0,\bar  y)$ is a {\em global space-time  minimizer}.  
\end{definition}

\subsection{The space-time embedding}
 Next lemma, being an easy consequence of the chain rule,  shows that the space-time problem $(P^e)$ restricted to the controls $(w_0,w)$ with $w_0>0$ a.e.   is  an equivalent formulation of the original problem  $(P)$.  Since  every $L^1$-equivalence class contains  Borel measurable representatives,  here  and  in the sequel we  tacitly assume that all $L^1$ maps   are Borel measurable,  when necessary.
 
  \begin{lemma}\label{cres} (i) If  $(T,u,x)$ is a strict sense process, then 
  $$
  {\mathcal I}(T,u,x):= (S,w_0,w,y^0,y),
 $$
 where $(S,w_0,w,y^0,y)$ is defined as in \eqref{sigma}, 
  is a space-time process  with  $w_0>0$   a.e.  on $[0,S]$.
 
 (ii)  Vice-versa, if $(S,w_0,w,y^0,y)$ is a space-time process with $w_0>0$   a.e.  on $[0,S]$,  then
  $$
  {\mathcal I}^{-1}(S,w_0,w,y^0,y):=(T,u,x),
 $$
where
$$
\begin{array}{l}
 \sigma(t):=(y^0)^{-1}(t),  \ \ T:=y^0(S), \\ [1.5ex]
 x(t) := y(\sigma(t)),  \quad \text{for } t\in[0,T], \\ [1.5ex]
 u(t):=\ds\frac{w(\sigma(t))}{w_0(\sigma(t))} \  \  \text{a.e. } t\in[0,T].
\end{array}
$$
  is a strict sense process. 
  
Furthermore, in  both cases one has $$\Psi(T,x(T))=\Psi((y^0,y)(S)).$$
\end{lemma}
 \vsm
Notice that the impulsive extension  consists in  allowing  subintervals  $I\subseteq [0,S]$  where $w_0\equiv0$. Then   the state $y$ evolves  on  $I$ in  zero  $t$-time, driven by  the non-drift  dynamics $\sum_{i=1}^{m}g_{i}(y(s))w^i(s)$. 
\vsm
\begin{remark} Let us point out that a reparameterization like the one utilized above is made possible by the fact that our localization of the problem implies that we are looking for minima among controls $u$ with  uniformly bounded $L^1$ norms. Of course, by relaxing this constraint,  larger classes of controls  can be considered. This, however, leads to the consideration of much more structured processes, in the direction  e.g. of  \cite{AR15, AMR15,MS18, BR94} or \cite{LQ}.
\end{remark}
\vsm
Actually,  the notion of  space-time  $L^{\infty}$-local minimizer  is consistent with the definition of strict-sense $L^{\infty}$-local minimizer, as stated in the  following result:

\begin{proposition}\label{Min=} A feasible strict sense process  $(\bar T, \bar u, \bar x)$  is a  strict sense $L^{\infty}$-local minimizer  for problem  $(P)$ if and only if   $(\bar S, \bar w^0,\bar w, \bar y^0, \bar y):={\mathcal I}(\bar T,\bar u,  \bar x)$
 is an $L^{\infty}$-local minimizer for problem $(P^e)$  among  the feasible space-time processes   $(S, w^0,w, y^0,y)$ with $w^0>0$ a.e..    Moreover,   
  $$
 \Psi((\bar y^0, \bar  y)(\bar S))=\Psi(\bar T,\bar x(\bar T)).
 $$
\end{proposition}
\begin{proof}  We have to prove the  following assertion:
\begin{itemize}
\item[(a)] there is  $\delta>0$ such that 
 $
\Psi(\bar T, \bar x(\bar T))\,\leq\, \Psi(T,  x(T))
 $
for all feasible  strict-sense processes  $(T, u, x)$ verifying  
\bel{close1App}
{\rm d}\Big((T, x,\va[u]),(\bar T,\bar x, \va[\bar u])\Big)
< \delta 
\eeq
if and only if
\item[(b)] there is  $\delta'>0$ such that 
$
\Psi((\bar y^0, \bar  y)(\bar S)) \leq  \Psi((y^0, y)(S))
$
 for all feasible space-time processes $(S, w^0,w, y^0, y)$  verifying  
 \begin{equation}
\label{close2App}
\begin{array}{l}
{\rm d}\Big(( y^0(S), y, \va[w]),( \bar y^0(\bar S),\bar  y,\va[\bar w])\Big)
< \delta'\,.
 \end{array}
\end{equation}
\end{itemize}
 
Let us show that  (a) $\Longrightarrow$ (b).  
For any  space-time process $(S,w^0,w, y^0, y)$ with $w^0>0$ a.e., let us consider the strict-sense process $(T, u, x):={\mathcal I}^{-1}(S,w^0,w, y^0, y)$, as defined in Lemma \ref{cres}. In the rest of the proof we still denote $x$, $\bar x$, $\va[u]$, $\va[\bar u]$, $y$, $\bar y$,  $\va[w]$, and $\va[\bar w]$   the constant continuous  extensions of these functions
  to $[0,+\infty)$.  Instead,  the maps $y^0$, $\bar y^0$ are extended to  $[0,+\infty)$ by setting --we  do not rename-- 
  $$
  y^0:=(T+s-S)\chi_{(S,+\infty)}, \quad \bar y^0:=\bar (T+s-\bar S)\chi_{(\bar S,+\infty)}.
  $$
Let  $\sigma$, $\bar\sigma:[0,+\infty)\to[0,+\infty)$ be the continuous,  increasing, one-to-one maps given by $\sigma:=(y^0)^{-1}$, $\bar \sigma:=(\bar y^0)^{-1}$. For every $t\ge0$,  we set 
$$
s:=\sigma(t), \quad\bar s:=\bar\sigma(t), 
$$
so that $|s-\bar s|=|\sigma(t)-\bar\sigma(t)|$ (and $y^0(s)=t=\bar y^0(\bar s)$). Recalling  the definition of  ${\mathcal I}$,  we have
$$
\begin{array}{l}
|\bar T-T|+ |(\bar x,\va[\bar u])(t)- (x,\va[u])(t)|  \\ [1.5ex]
\qquad \quad =|\bar y^0(\bar S)-y^0(S)|+|(\bar y, \va[\bar w])(\bar s)-(y,\va[w])(s)|.
\end{array}
$$
Moreover, 
$$
\begin{array}{l}
\va[w](s)-\va[\bar w](\bar s)= s-\bar s+ \bar y^0(\bar s)-y^0(s)=s-\bar s
\end{array}
$$
and   the Lipschitz continuity of $\bar y$  also implies that
\bel{dtr}
|\bar y(\bar s)-y(s)|\le L|\bar s-s|+ |\bar y(s)-y(s)| 
\eeq
for some $L>0$.  Let us set $\bar t:=\bar y^0(s)$, so that $\bar\sigma(\bar t)=s=\sigma(t)$. Then
 \begin{equation*}
 \begin{split}
  |s-\bar s| &=|\sigma(t)-\bar\sigma(t)|=|\bar\sigma(\bar t)-\bar\sigma(t)|\le\omega_{\bar\sigma}(|\bar t-t|) \\ 
&= \omega_{\bar\sigma}(|\bar y^0(s)-y^0(s)|)= \omega_{\bar\sigma}( |\va[\bar w](s)-\va[w](s)|),
  \end{split}
  \end{equation*}
where $\omega_{\bar\sigma}$ is the modulus of continuity of $\bar\sigma$, which exists since $\bar\sigma=(\bar\varphi^0)^{-1}$ is absolutely continuous and thus uniformly continuous. In conclusion, we obtain
\begin{equation*}
\begin{split}
|\bar T-T|+ |(\bar x,\va[\bar u])&(t)- (x,\va[u])(t)|  \\
&\le |\bar y^0(\bar S)-y^0(S)|+ \|\bar y-y\|_\infty   \\ 
& \quad\quad+(1+L) \omega_{\bar\sigma}(\| \va[\bar w] -\va[w] \|_\infty),
\end{split}
\end{equation*}
 which implies assertion (b), as soon as we choose $\delta'>0$ verifying $2\delta'+(1+L)\omega_{\bar\sigma}(\delta')<\delta$. Indeed, we obtain
$$
\begin{array}{l}
\Psi((\bar y^0, \bar  y)(\bar S) )=\Psi(\bar T, \bar x(\bar T)) \\ [1.5ex]
\qquad\qquad\qquad \leq\Psi(T,  x(T))=\Psi((y^0,y)(S)),
 \end{array}
$$ 
  for all feasible space-time processes  $(S, w^0,w, y^0, y)$ with $w^0>0$ a.e.  and   satisfying \eqref{close2App}  for such $\delta'$. 

\vsm
Let us now prove that (b) $\Longrightarrow$ (a).   For  any feasible  strict sense process  $(T, u, x)$ let us set
$(S,w^0,w, y^0, y):={\mathcal I}(T,u, x)$. Once again, we consider the functions extended to $[0,+\infty)$ as described above.
 
For any $s\ge0$, let us set  $t:=y^0(s)$ and    $\bar s:=(\bar y^0)^{-1}(t)$.  
We define  $\sigma :=(y^0)^{-1}$ and $\bar\sigma:=(\bar y^0)^{-1}$. Then, using   $L>0$ to  denote the Lipschitz constant of $\bar y$,   we get
$$
\begin{array}{l}
 |\bar y^0(\bar S)-y^0(S)|+|(\bar y, \va[\bar w])(s)-(y,\va[w])(s)| \\ [1.5ex]
 \qquad \le|\bar T-T|+|(\bar y, \va[\bar w])(s)- (\bar y, \va[\bar w])(\bar s)|  \\ [1.5ex]
\qquad\qquad\qquad\qquad + |(\bar y, \va[\bar w])(\bar s)-(y,\va[w])(s)|  \\ [1.5ex]
\le |\bar T-T|+(1+L)(| \va[\bar u](t)-\va[ u](t)|)+ |\bar x(t)-  x(t)| ,
\end{array}
$$
where the last inequality holds,  because
$$
 |\bar s-s|=|\bar\sigma(t)-\sigma(t)|= |\va[\bar u](t)-\va[ u](t)|.
$$
At this point, we derive assertion (a) as soon as we choose $\delta$ such that $2\delta+(1+L)\delta<\delta'$, since
$$
\begin{array}{l}
\Psi(\bar T, \bar x(\bar T))=\Psi((\bar y^0, \bar  y)(\bar S) )  \\
\quad\qquad\qquad\le \Psi((y^0,y)(S))=\Psi(T,  x(T)),
\end{array}
$$ 
  for all strict sense feasible processes  verifying \eqref{close1App}  for such $\delta$. 
  \end{proof}

 \section{A HIGHER ORDER MAXIMUM PRINCIPLE}\label{SNC}
Let us consider the  {\it unmaximized Hamiltonian}
$$
H(x,p_0,p,\lambda,w^0,w):= p_0w^0 + p\cdot\Big(f(x) w^0 +  \sum_{i=1}^{m}  g_{i}(x) w^i\Big)
$$
 and the  {\em  Hamiltonian} ${\bf H}: \cR^n\times \cR\times\cR^n\times\cR\to \cR$ defined by 
$${\bf H}(x,p_0,p, \lambda) := \ds\max_{(w^0,w)\in \mathcal{W}}  H(x,p_0,p ,\lambda,w^0,w ),
$$
where  $ \mathcal{W}$ is as in \eqref{DS+}. 
For any continuous vector field $F:\cR^n\to\cR^n$, let us introduce the classical {\it $F$-Hamiltonian}
  $$
  {\bf H}_F(x,p):=p\cdot F(x) \qquad\text{for } (x,p)\in\cR^n\times\cR^n.
  $$
   Furthermore, let us define the 
 the {\it polar cone} of $\cS$ at a point  $(T,x)\in \cS$ as the set
 \begin{multline*}
  N_{_{(T,x)}}\cS:=
\ds\text{span}_+\Big\{D\varphi_\ell(T,x): \,\,\,\ell\in  I(T,x)\Big\} \\ 
+ \text{span}\Big\{D\psi_j(T,x):\,\,\, j=1\,\ldots,r_2 \Big\},
  \end{multline*}
  where $I(T,x)\subseteq\{1,\ldots,r_1\}$ is the subsets  of indexes $\ell$ such that $\varphi_\ell(T,x) = 0$.

  
 \begin{theorem}\label{PMPeho}[{\sc Higher Order Maximum Principle}] Let  $(\bar S,\bar w^0,\bar w,  \bar y^0,\bar  y)$  be an $L^\infty$-local minimizer 
for  the space-time problem $(P^e)$. Then
there exists  a multiplier   $(p_0, p,\lambda)\in \cR\times AC\left([0,\bar S],\cR^{ n}\right)\times [0,+\infty)$ such that the following conditions hold true:
 
\begin{itemize}
\item[(i)] {\sc (non-triviality)} 
\begin{equation} 
\label{fe1}
(p_0, p , \lambda) \not= (0, 0,0) \,;
\end{equation}
  \item[(ii)] {\sc (non-tranversality)}  
\begin{equation}\label{fe4}
 (p_0,p(\bar S)) \in -\lambda D\Psi\Big(( \bar y^0, \bar y)(\bar S)\Big)- N_{( \bar y^0, \bar y)(\bar S)}\cS;
 \eeq
\item[(iii)] {\sc (adjoint equation)}  the path $p$ solves on $[0,\bar S]$  the adjoint equation
\begin{equation}
\label{fe2}
\displaystyle  \frac{dp}{ds} =\,- p \cdot \bigg( Df  ( \bar y )) \bar w^0 +  \sum_{i=1}^{m}Dg_{i} (\bar y) \bar w^i\bigg);
\end{equation} 
\item[(iv)] {\sc (First order maximization)} 
 for  a. e. $s\in [0,\bar S]$,  one has
\begin{equation}\label{fe3}
\begin{array}{l}
H\Big(\bar y(s), p_0 , p(s) , \lambda,\bar w^0(s),\bar w(s) \Big)=  \\
 \qquad\qquad\qquad\qquad\qquad {\bf H}\Big(\bar y(s), p_0 , p(s), \lambda\Big);
\end{array}
\eeq
\item[(v)] {\sc (Vanishing of Hamiltonians)} 
\bel{engine}
 {\bf H}\Big(\bar y(s), p_0 , p(s),  \lambda\Big)=0 \qquad  \text{for all } s\in [0,\bar S];
\eeq
 \bel{pg0hi}
 {\bf H}_{g_i}(\bar y(s),p(s))=0 \   \text{for all }  s\in [0,\bar S], \  i=1,\dots,m;
 \eeq
 \item[(vi)] {\sc (Vanishing of higher order Hamiltonians)} 
\bel{pg000hi}
  {\bf H}_{B}(\bar y(s),p(s))= 0   \qquad\text{for all } s\in [0,\bar S], 
  \eeq
  for every iterated bracket $B$ of the vector fields $g_1,\dots,g_m$.
\end{itemize}
 Furthermore, if the trajectory $\bar y $ is not instantaneous, namely, if 
   $ \bar y^0(\bar S)>0$, then \eqref{fe1} can be strengthened  to 
\begin{equation}\label{strongfe1}
  (p  , \lambda) \not= (0,0)\,\,. 
  \end{equation}
\end{theorem}
\vsm
The existence of a  multiplier verifying the first order conditions (i)--(v) and   of the  strengthened non-triviality condition \eqref{strongfe1}  has been already proved in \cite{MRV18}  as  direct  consequence  of the standard Maximum Principle.
Instead, the fact that the same multiplier verifies the higher order relations in  (vi)
needs a proof that exceeds the space limits of  the present paper. A proof of a stronger version of this theorem, including very low regularity assumptions on both the vector fields and the target,  can be found  in  \cite{AMR19}.
\vsm
\begin{remark}\label{Rext} {\rm The higher order condition \eqref{pg000hi} in (vi)
has been obtained in  \cite{ChiSt16} for the special case of non-impulsive (but unbounded) optimal time trajectories. We are able to prove it for possibly impulsive trajectories due to the fact that one can construct {\it instantaneous} approximations of Lie brackets. Incidentally let us remark that, unlike what is done in \cite{ChiSt16},   we do not assume the constancy of the rank of the Lie Algebra generated by $g_1,\ldots,g_{m}$. Actually, in the result proved in \cite{AMR19}
we do not even assume that the vector fields $g_1,\ldots,g_{m}$ are $C^\infty$: the only regularity required is the continuity of the involved Lie bracket $B$.} 
\end{remark}
\vsm
We get immediately further higher order conditions involving the drift $f$ as well.
\begin{corollary}\label{COR}
 Let  $(\bar S,\bar w^0,\bar w,  \bar y^0,\bar  y)$  be an $L^\infty$-local minimizer 
for  the space-time problem $(P^e)$. Then
there exists  a multiplier   $(p_0, p,\lambda)\in \cR\times AC\left([0,\bar S],\cR^{ n}\right)\times [0,+\infty)$ such that, besides verifying conditions 
conditions (i)-(vi)  of Theorem \ref{PMPeho} makes the relation 
  \bel{pg000hif}
  {\bf H}_{[f,B]}(\bar y(s),p(s))\bar w^0(s)= 0   \quad\text{a.e.  } \ s\in [0,\bar S], 
  \eeq
 hold true  for every iterated bracket $B$ of the vector fields $g_1,\dots,g_m$.
\end{corollary}

\begin{proof}
Condition  \eqref{pg000hif} can be obtained by simply differentiating   \eqref{pg000hi}  and  recalling  that the derivative of $p$ verifies \eqref{fe2}. Indeed, for a. e.  $s\in[0,\bar S]$ one has
$$
 \begin{array}{l}\ds \frac{d }{ds}\left(p(s)\cdot B(\bar y(s))\right)    \\ [1.5ex]
 \qquad\qquad\ds=\frac{d p}{ds}(s)\cdot B(\bar y(s))+p(s)\cdot D\, B(\bar y(s)) \frac{d\bar y}{ds}(s)  \\ [1.5ex]
 \qquad\qquad = p(s)\cdot  \bigg([f,B]( \bar y(s)) \bar w^0(s)+ \ds  \sum_{i=1}^{m} [g_i,B] (\bar y(s))\bigg)   \\ [1.5ex]
 \qquad\qquad =p(s)\cdot  [f,B]( \bar y(s)) \bar w^0(s)=0.
 \end{array}
  $$
\end{proof} 
 

 \section{AN EXAMPLE}\label{SectionExample}
This example  shows how higher order necessary conditions may be useful to rule out the optimality of a space-time process  for which there exists a multiplier verifying all the  first order conditions (i)--(v) of Theorem \ref{PMPeho}, namely, the usual maximum principle.
\vsm
Consider the problem
\begin{equation}\label{exampleP}
\left\{
\begin{array}{l}
\;\;\;\;\;\;\;\;\;\;\;\;\mbox{Minimize }\  \Psi(x(1)),
\\ [1.5ex]
\ds\mbox{over  } \ (u,x):[0,1]\to \cR^{2}\times\cR^{5} \   \mbox{s. t. }  \\ [1.5ex]
\displaystyle\frac{dx}{dt} \,=\, f(x) + \sum_{i=1}^{2}g_{i}(x)u^i   \  \mbox{ a.e. } t \in [0,1],  \\ [1.5ex]
x(0) = (1,0,0,0,0), \quad  x(1) \in \cS,
\end{array}
\right.
\end{equation}
in which   $\Psi(x):=x^3+x^4$, 
$$
\cS:=\{(x^1,\dots,x^5)\in\cR^5: \ x^1=x^2=0, \ x^3\ge-1\}, 
$$
$$
\begin{array}{l}
\ds f(x): = \frac12\Big( (x^2)^2 + (x^3)^2+(1-x^1-x^5)^2\Big) \frac{\partial\ }{\partial x^4}+ \frac{\partial\ }{\partial x^5},  \\ [1.5ex]
\ds g_1(x):= \frac{\partial\ }{\partial x^1} -\frac12 (x^2)^2\frac{\partial\ }{\partial x^3}, \qquad  
g_2(x):=\frac{\partial\ }{\partial x^2}.
\end{array}
$$ 
The corresponding space-time problem reads
\begin{equation}\label{exampleSTS}
\left\{
\begin{array}{l}
\;\;\;\;\;\;\;\;\;\;\;\;\mbox{Minimize }\  \Psi(y(S)),
\\ [1.5ex]
\ds\mbox{over  } S>0, \ (w^0,w,y^0,y):[0,S]\to \cR^{9} \   \mbox{s. t. }  \\ [1.5ex]
\displaystyle \frac{dy^0}{ds} = w^0,\\ [1.5ex]
\displaystyle \frac{dy}{ds} = f(y)w^0 + \sum_{i=1}^2 g_i(y) w^i  \  \mbox{ a. e. } s \in [0,S],  \\ [1.5ex]
y^0(0)=0,    \ y(0)=(1,0,0,0,0),\     (y^0 , y)(S) \in \{1\}\times \cS.
\end{array}
\right.
\end{equation}

Let us  consider  the feasible  space-time  process $(\hat S, \hat w^0,\hat w,\hat y^0,\hat y)$, where $\hat S=\sqrt{2}$, 
$$
 (\hat  w^0,\hat {w}^1, \hat {w}^2)= \left(\frac{\sqrt{2}}{2},-\frac{\sqrt{2}}{2},0\right),
 $$
 and
 $$
 (\hat y^0,\hat y^1,\dots, \hat y^5)=  \left(\frac{\sqrt{2}\,s }{2},1-\frac{\sqrt{2}\,s}{2},0,0,0,\frac{\sqrt{2}\,s}{2}\right) 
$$
on $[0,\sqrt{2}]$. 
 In the sequel,  we prove that $(\hat S, \hat w^0,\hat w,\hat y^0,\hat y)$ is an {\it extremal} for the space-time problem \eqref{exampleSTS} --that is, it verifies conditions (i)--(v) of Theorem \ref{PMPeho} for some  multiplier--    but there is no non-zero multiplier  for which all the necessary conditions in Theorem \ref{PMPeho} are met. 
 
 
The adjoint equation and the non-transversality condition read
$$
\left\{
\begin{array}{l}
\ds\frac{dp_0}{ds} = 0, \\
\ds\frac{dp_1}{ds} = p_4(1-\hat y^1-\hat y^5) \hat{w}^0 \\
\ds\frac{dp_2}{ds} =p_3 \hat y^2  \hat{w}^1-p_4 \hat y^2  \hat{w}^0 \\
\ds\frac{dp_3}{ds} =-p_4 \hat y^3  \hat{w}^0 \\
\ds\frac{dp_4}{ds} =0 \\ 
\ds\frac{dp_5}{ds} =p_4(1-\hat y^1-\hat y^5) \hat{w}^0 \\ 
\end{array}\right.\qquad \text{a.e. } s\in[0,\sqrt{2}],
$$
$$
\begin{array}{l}
(p_0,p(\sqrt{2}))
= - \lambda (0,0,0,1,1,0) - \cR^3 \times  \{(0,0,0)\} ,
 \end{array}
$$
with $\lambda\geq 0$.  Therefore,   $p_0$, $p_1$, and $p_3$ are arbitrary real constants, while $p_3=p_4=-\lambda$  and $p_5=0$.   Moreover, by (the first order conditions)   \eqref{fe3} and \eqref{engine} we get  
$$
p_0\hat{w}^0+p_1\hat{w}^1+p_2\hat{w}^2=p_0\,\frac{\sqrt{2}}{2}-p_1\,\frac{\sqrt{2}}{2}=0
$$
and
$$
p_0 {w}^0+p_1 {w}^1+p_2 {w}^2\le0  
$$
for all $(w^0,w^1,w^2)\in[0,+\infty)\times\cR^2$ with $w^0+|w|=1$. 
Hence there exists a non-trivial multiplier $(p_0,p,\lambda)$,   necessarily of the form
$$
 (p_0,p,\lambda)=(0,0,0,-\lambda, -\lambda, 0, \lambda), \qquad \lambda>0.
 $$
However, since 
$$
\big[[g_1,g_2],g_2\big] = -\frac{\partial\ }{\partial x^3},
$$
the higher-order condition \eqref{pg000hi} implies $p_3\equiv0,$ thus $\lambda=0,$ and then $(p_0,p,\lambda)=(0,0,0)$.  Therefore, we can   conclude that  $(\hat S, \hat w^0,\hat w,\hat y^0,\hat y)$ is not a minimizer, since there is no non-zero  multiplier verifying the higher order necessary conditions in Theorem \ref{PMPeho}.

Actually,  it is not difficult to see that the space-time process $(\bar S, \bar w^0,\bar w, \bar y^0,\bar y)$, where $\bar S=\sqrt{2}+4\sqrt[3]{2}$, 
$$
I_0:=[0,\sqrt{2}], \  I_i:=\sqrt{2}+((j-1)\sqrt[3]{2},j\sqrt[3]{2}], \ \  j=1,\dots,4,
$$
$$
 (\bar w^0, \bar{w}^1, \bar{w}^2)= \left\{
\begin{array}{l}
\left(\frac{\sqrt{2}}{2},-\frac{\sqrt{2}}{2},0\right) \qquad    I_0\\  
 (0,-1,0) \qquad \qquad    I_1  \\ 
(0,0,-1)  \qquad \qquad   I_2\\
 (0,1,0)  \qquad \qquad  \ \   I_3 \\ 
(0,0,1) \qquad \qquad  \ \   I_4 
\end{array}
\right.
$$
and
\,

$
(\bar y^0,\bar y^1,\dots, \bar y^5) =
$
$$
 {\small \left\{
\begin{array}{l}
\ds\left(\frac{\sqrt{2}}{2}\,s,1-\frac{\sqrt{2}}{2}\,s,0,0,0,\frac{\sqrt{2}}{2}\,s\right) \qquad\qquad \qquad\qquad\qquad \  I_0\\
\ds \left(1,\sqrt{2}-s,0,0,0,1\right)\qquad\qquad \qquad\qquad\qquad\qquad\qquad\ \  \ \, I_1 \\
\ds\left(1,-\sqrt[3]{2}, \sqrt{2}+\sqrt[3]{2}-s,0,0,1\right) \qquad\qquad \qquad\qquad\qquad  \ \  \, I_2\\
\ds\left(1, s-\sqrt{2}-3\sqrt[3]{2},-\sqrt[3]{2},\frac{(\sqrt[3]{2})^2}{2}(\sqrt{2}+2\sqrt[3]{2}-s) ,0,1\right)  \     I_3 \\
\ds \left(1, 0, s-\sqrt{2}-4\sqrt[3]{2},-1,0,1\right) \qquad\qquad \qquad\qquad\qquad \quad   \, I_4,
\end{array}
\right.}
$$
is a global minimizer.  In particular,  this  process steers the state $(1,0,0,0,0)$ to $(0,0,0,0,1)$ by a uniform rectilinear  motion 
 until time $t=1$. After that, the state jumps instantly to $(0,0,-1,0,1)$. 
 
 The non-transversality condition \eqref{fe4} now reads
$$
\begin{array}{l}
\ds \left(p_0,p(\sqrt{2}+4\sqrt[3]{2})\right) \\
 \qquad\quad =- \lambda (0,0,0,1,1,0) - \cR^3 \times (-\infty,0]\times \{(0,0)\},
 \end{array}
$$
with $\lambda\ge0$.  This yields 
$$
\ds \left(p_0,p(\sqrt{2}+4\sqrt[3]{2})\right)=(c_0,c_1,c_2,-\lambda+c_3,-\lambda,0)
$$
with $c_0$, $c_1$, $c_2\in\cR$ and $c_3\ge0$. Choosing $c_0=c_1=c_2=0$, $c_3=1$ and $\lambda=1$, we get the non-trivial,  constant multiplier 
$$
(p_0,p_1,p_2,p_3,p_4,p_5,\lambda)=(0,0,0,0,-1,0,1),
$$
which  satisfies all the necessary conditions in Theorem \ref{PMPeho} (and  agrees with the strengthened non-triviality condition \eqref{strongfe1}, i.e. $(p,\lambda)\ne(0,0)$,  that is in force,  for $\bar y^0(\sqrt{2}+4\sqrt[3]{2})=1>0$). In particular, the higher order condition  \eqref{pg000hi} is trivially verified, since the vector fields $g_1$, $g_2$, and all the elements of the Lie algebra generated by $\{g_1,g_2\}$ have
  the fourth component equal to zero.

\addtolength{\textheight}{-12cm}   
 

\end{document}